\documentclass[12pt,a4paper]{amsart}

\usepackage{caption}
\usepackage{longtable}
\usepackage{graphicx}
\usepackage{epsfig}
\usepackage{color}
\usepackage{array}
\usepackage{amsmath}
\usepackage{amsthm}
\usepackage{amssymb}
\usepackage{enumerate}
\usepackage{amscd}
\usepackage[bookmarksnumbered,bookmarksopen,pdfusetitle,unicode]{hyperref}
\usepackage[all,cmtip]{xy}

\numberwithin{equation}{section}
\theoremstyle{plain}
\newtheorem{thm}{Theorem}[section]

\newtheorem{lem}[thm]{Lemma}
\newtheorem{prop}[thm]{Proposition}
\newtheorem{conj}[thm]{Conjecture}
\newtheorem{propdef}[thm]{Proposition-Definition}
\theoremstyle{definition}
\newtheorem{defn}[thm]{Definition}

\newtheorem{ex}[thm]{Example}

\theoremstyle{remark}
\newtheorem{rem}[thm]{Remark}

\def\ZZ{\mathbb{Z}}

\newcommand{\N}{\mathbb N}
\newcommand{\R}{\mathbb R}
\newcommand{\CC}{\mathbb{C}}

\newcommand{\supp}[1]{\textrm{supp}({#1})}
\usepackage{subfig}
\title[Nonisolated forms of rational triple points]{Nonisolated forms of rational triple point singularities of surfaces and their resolutions}

\author{A. Alt{\i}nta{\c s}}
\address{Department of Mathematics, Y{\i}ld{\i}z Technical University\\ Esenler 34210, Istanbul, Turkey} 
 \email{aysealtintas@gmail.com}
 
\author{G. {\c C}ev{\.i}k}
\address{Department of Mathematics, Ko{\c c} University\\ Sar{\i}yer 34450, Istanbul, Turkey}
\email{gulencevik@gmail.com}
\thanks{The second and third authors were supported by the project 109T667 under the program 1001 of the Scientific and Technological Research Council of Turkey.}
\author{M. Tosun}
\address{Department of Mathematics, Galatasaray University\\  Ortak{\"o}y 34357, Istanbul, Turkey}
\email{mtosun@gsu.edu.tr}

\subjclass[2000]{58K20}

\begin{document}

\maketitle
\tableofcontents

\section{Introduction}

Let $S$ be a two-dimensional normal analytic space embedded in $\mathbb C^N$ having an isolated singularity at the origin. Let  $\pi\colon \tilde{S}\longrightarrow S$ be a resolution of $S$. The singularity of $S$ is called rational if $H^1(\tilde{S},{\mathcal O}_{\tilde{S}})=0$. This condition implies very nice combinatorial results on dual resolution graphs of rational singularities (\cite{Ar}).  For example, the multiplicity of a rational singularity equals $-Z^2$ where $Z$ is the Artin's divisor (see Section 2) supported on the minimal resolution graph of the singularity.  Moreover, a rational singularity of multiplicity $m$ can be given by $m(m-1)/2$ equations with linearly independent quadratic terms (\cite{wahl}). The rational singularities of multiplicity $2$ are famously known as rational double (\textit{RDP}) or Du Val singularities (see, for example, \cite{barth}). In \cite{Ar}, M. Artin gave the complete list of the minimal resolution graphs of rational singularities of surfaces of multiplicity $3$ (\textit{rational triple point singularities} or \textit{RTP-singularities}, for short). We will recall those graphs in Table \ref{tablo-res}. The list of minimal resolution graphs for multiplicity $4$ and $5$ were given  in \cite{stevens1}  and \cite{t-o-o} respectively.  Those graphs were classified by using the combinatorics of the dual resolution graphs. The classification problem of rational singularities by their minimal graphs was studied deeply in  \cite{le-tosun} and \cite{stevens2}.
 
In \cite{tjurina}, Tjurina proved that rational singularities of surfaces are absolutely isolated, i.e. can be resolved by blowing up without normalisation, and gave a list of explicit equations defining the RTP-singularities. Her construction is based on the fact that a subgraph of a resolution graph of a rational singularity is still a resolution graph of a rational singularity (\cite{Ar}).  According to \cite{tjurina},  a surface  having an RTP-singularity is defined by $3$ equations in $\mathbb{C}^4$ (see also Section \ref{sec-rtp}).  So, they are neither hypersurface singularities nor complete intersection singularities. This makes them one of the most interesting objects in Singularity Theory/Algebraic geometry as they provide examples in a better understanding of other singularities of surfaces.

In this work, we study the equations defining RTP-singularities and give a new construction of their minimal resolution graphs. Our presentation is divided into three main sections. After recalling some basic facts about rational singularities of surfaces, we recall a global construction of triple covers from algebraic geometers point of view in Section \ref{sec-mir}. Using the fact that any normal
surface singularity is the normalisation of a nonisolated hypersurface singularity, we obtain explicit equations of some nonisolated hypersurfaces in $\CC^3$ whose normalisations give the RTP-singularities. Since the normalisations of our equations exactly correspond to the ones listed by Tjurina  (see Proposition \ref{prop-hyper}) we will call them {\it nonisolated forms of RTP-singularities}. 

A list for nonisolated forms of RTP-singularities were also obtained in \cite{tan-etal} by a different construction and some of the equations (such as $A_{k-1,\ell-1,m-1}$, $C_{\ell+1,k-1}$ and $F_{k-1}$) differ from ours.  Their construction is based on \cite{tan} where the author studied triple covers $Y\rightarrow X$ by global data on $X$ using the classical method of solving cubic equations and presented conditions for the cover to be smooth with smooth branch locus and other properties. In the case of surfaces, that technique provides a resolution of singularities of both the branch locus and of $Y$. This method is in fact called the Jung's resolution of singularities, studied in \cite{laufer} and \cite{le-bondil}. 

The cubic equations of nonisolated forms of RTP-singularities listed here may not have the simplest forms but are obtained by the suitable projections for our purposes in Section \ref{sec-rtp}. There, we construct an abstract graph from a arbitrary polygon in $\R^3$ by a regular subdivision and show that it may not correspond to a resolution graph of a singularity if it is not a Newton polygon (see Section \ref{sect-algo} and Remark \ref{rem-nonnew}). Then we construct the resolution graphs of RTP-singularities using the Newton polygons of those cubic equations. This method is given in \cite{oka} in the case of non-degenerate complete intersection singularities. Here, we simplify the method (for example, Definition 5.17 which comes from Tropical Geometry), and refer to it as {\it  Oka's resolution process}. Even though many results in \cite{oka} concern complete intersection singularities, some of them contain the ``isolated singularity" hypothesis (e.g.  \cite[Theorem 6.2]{oka}) and no nonisolated examples were presented there. The equations that we give here are the first examples in the literature of nonisolated hypersurface singularities for which Oka's resolution process works. 

In the final part, we show that both normal equations and nonisolated forms are non-degenerate which means that they can be resolved by toric modifications associated with the regular subdivison of the corresponding Newton polygon.  This fact was shown in \cite{varchenko} for isolated hypersurface singularities and generalised in \cite{oka}. In Appendix, we recall a more general definition of non-degeneracy given in \cite{aroca}, where it was proved that all non-degenerate singularities  can be resolved by toric modifications, to show that the RTP-singularities are non-degenerate. This interesting property leads us to ask whether a singularity is non-degenerate if and only if its normalisation is non-degenerate.

%%%%%%%%%%%%%%%%%%%%%%%%
\section{Preliminaries on Rational Singularities}
%%%%%%%%%%%%%%%%%%%%%%%%

Assume that $(S,0)$ is a normal surface singularity embedded in $({\CC}^n,0)$ which means that the local ring ${\mathcal O}_{S,0}$ is normal. A \textit{resolution} of $(S,0)$ is a proper map $\pi\colon (\tilde{S}, E)\longrightarrow (S,0)$ such that $\tilde{S}$ is a nonsingular surface, $E:=\pi^{-1}(0)$ and the restriction of $\pi$ to $\pi^{-1}(S-{0})$ is an isomorphism.  The fibre $E$ is called  {\it the exceptional divisor} of $\pi $ which is, by the Zariski's Main theorem (\cite[Theorem V.5.2]{hartshorne}), a connected curve. A resolution $\pi$ is called {\it minimal} if any other resolution of $(S,0)$ factorizes via $\pi$. The minimal resolution exists and is unique. 
 
If the singularity $(S,0)$ is not isolated then first we apply a \textit{normalisation} $n\colon (\bar{S},0) \rightarrow (S,0)$ where $\bar{S}$ is a normal surface, $n$ is a finite and proper map. 
%it is resolved by first applying a \textit{normalisation} $n\colon (\bar{S},0) \rightarrow (S,0)$, where $\bar{S}$ is a normal surface, $n$ is a finite and proper map, and then a resolution defined as above. 

\begin{thm}[\cite{zariski}]
Any reduced complex surface admits a resolution.
\end{thm}

\begin{defn} A surface singularity is called rational if $H^1(\tilde{S},{\mathcal O}_{\tilde{S}})$=0. 
\end{defn}

Note that this characterisation of the rational singularities is independent of the choice of the resolution. The exceptional divisor of a resolution of a rational singularity is a normal crossing divisor of which each component $E_i$ is a nonsingular rational curve and its resolution graph is a tree (see, for example, \cite{tosun-turkish}).  Moreover,  by \cite{tjurina}, rational singularities can be resolved by a finite number blowing-ups (without normalisation). We also have a combinatorial description of rational singularities.

\begin{thm} [\cite{Ar}]
\label{art1}
 A normal surface singularity $(S,0)$ is rational if and only if $p_a(Y)\leq 0$ for any resolution $\pi\colon (\tilde{S},E) \rightarrow (S,0)$ where $p_a(Y)$ is the arithmetic genus of the positive divisor $Y:=\sum a_iE_i$, (i.e. $a_i\geq 0$ for all $i$), supported on the exceptional divisor $E=\cup_{i=1}^{n} E_i$.
\end{thm}

Moreover, among all positive divisors $Y$ supported on the exceptional divisor $E$ such that $(Y\cdot E_i)\leq 0$ for all $i=1,\ldots ,n$ there exists a smallest divisor, called \textit{Artin's divisor} of the resolution
$\pi $ and denoted by $Z$.  We have $Y:=\sum a_iE_i \geq Y':=\sum a'_iE_i$  if $a_i\geq a'_i$ for all $i=1,\ldots ,n$  with $a_i\geq 0$ and $a'_i\geq 0$. One of the information we get from Artin's divisor is the multiplicity of  the corresponding rational singularity. 

The \textit{multiplicity} of $S$ at $0$ is defined as the number of intersection points of $S$ by a generic affine space of codimension $2$ closed to the origin. It plays a key role in the study of singularities. In the case of rational singularities, it can be read from the resolution.

\begin{prop}[\cite{Ar}]
\label{art2}
Let $(S,0)$ be a rational singularity of multiplciy $m$. Then $Z^2=-m$ and the embedding dimension equals $m+1$.
\end{prop}

Conversely, if a given graph is weighted by $(w_i,g_i)$ at each vertex such that it is a tree with $g_i=0$ for all $i$ and satisfies the assertions of  Theorem \ref{art1} and Proposition \ref{art2}, then it is a resolution graph of a rational singularity. Then $w_i$ and $g_i$ represent the numbers $-E_i^2$ and  the genus of the corresponding irreducible component $E_i$ in the exceptional fibre, respectively.

%%%%%%%%%%%%%%%%%%%%%%%%%
\section{Triple Covers after Miranda}\label{sec-mir}
%%%%%%%%%%%%%%%%%%%%%%%%%

In \cite{mir-triple}, Miranda showed that a set of data for a \textit{triple cover} $p\colon Y\rightarrow X$ between two algebraic schemes (over characteristic $\neq 2,3$) consists of a free $\mathcal{O}_X$-module $\mathcal{E}$ of rank $2$ (say, generated by $z$ and $w$) such that $p_*\mathcal{O}_Y\cong \mathcal{O}_X\oplus \mathcal{E}$ and a morphism $\phi\colon S^2\mathcal{E}\rightarrow \mathcal{O}_X\oplus \mathcal{E}$ given by
\begin{eqnarray}\label{phi3}  \phi(z^2)&=&2(a^2-bd)+az+bw, \nonumber \\
 \phi(zw)&=&-(ad-bc)-dz-aw, \nonumber \\
 \phi(w^2)&=& 2(d^2-ac)+cz+dw \nonumber \end{eqnarray} where $S^2\mathcal{E}$ is the second symmetric power of $\mathcal{E}$ and $a,b,c,d\in \mathcal{O}_X$ with $bc\neq 0$. Here we remark that Miranda's construction also works for the case where $X$ and $Y$ are germs of analytic varieties even if $\mathcal{E}$ might fail to be a free $\mathcal{O}_X$-module. 

Let $p\colon Y\rightarrow X$ be a covering map of degree $3$ between two analytic varieties $X$ and $Y$ with $p_*\mathcal{O}_Y\cong \mathcal{O}_{X}\cdot \{1,z,w\}$. Then one can write
\begin{eqnarray}\label{phi1} z^2&=&g+az+bw, \nonumber \\
zw&=&h+ez+fw, \\
w^2&=& i+cz+dw \nonumber \end{eqnarray}
for $a,b,\ldots, i\in \mathcal{O}_X$.  Multiplying the equations in (\ref{phi1}) by $w$, $z$ and $w$ respectively we get
\begin{equation} \label{eqasso} g=be+f^2-af-bd,\hskip6pt h=bc-ef, \hskip6pt i=e^2+cf-ac-de \end{equation} 
since $z\cdot zw = w\cdot z^2$ and $z\cdot w^2=w\cdot zw$ (cf. \cite[Lemma 2.4]{mir-triple}). By the set up, no cubic polynomial in $z$ and $w$ has a square term, $z^3$ is generated by $1$ and $z$ in $\mathcal{O}_X$ and similarly, $w^3$ by $1$ and $z$ in $\mathcal{O}_X$. So, (\ref{phi1}) and (\ref{eqasso}) give
\begin{eqnarray*} z^3&=& ag+bh+(g+a^2+be)z+(ab+bf)w ,\\
w^3&=& ch+di+(ce+cd)z+(i+cf+d^2)w.
\end{eqnarray*}
Therefore, $ab+bf=0$ and $ce+cd=0$ on $\mathcal{O}_X$ which yield $f=-a$ and $e=-d$. Because,  when $b=0$ (resp. $c=0$) we have $z^2=g+az$ (resp. $w^2=i+dw$); this contradicts the fact that the field of fractions $K_Y$ over $\mathcal{O}_Y$ is an extension of $K_X$ of degree 3  (cf. \cite[Lemma 2.6]{mir-triple}).

Now, let us consider a triple cover $p\colon Y \rightarrow X$ where $X$ is smooth and $Y$ is defined by
\begin{eqnarray}\label{defY} F(z,w)&:=&z^2-2(a^2-bd)-az-bw, \nonumber \\
G(z,w)&:=&zw+(ad-bc)+dz+aw, \\
H(z,w)&:=&w^2-2(d^2-ac)-cz-dw \nonumber \end{eqnarray}
with $a,b,c,d\in \mathcal{O}_X$.

\begin{prop}  (see also \cite{wahl}) With preceding notations, the embedding of $Y$ into $\mathbb{C}^2\times X$ given by (\ref{defY}) is determinantal.
\end{prop}

\begin{proof} Recall that a (germ of an analytic) variety $V\subseteq \mathbb{C}^N$ is said to be \textit{determinantal} if its defining ideal is generated by the $(t\times t)$-minors of an $(r\times s)$-matrix over $\mathcal{O}_{\mathbb{C}^N,0}$ for $0<t\leq r\leq s$ and $\textnormal{codim}(V)=(r-t+1)(s-t+1)$. The affirmation easily follows since the codimension of $Y$ in $\mathbb{C}^2\times X$ is $2$  and the polynomials $F,G,H$ above can be written as the $(2\times 2)$-minors of the matrix
\begin{equation}\label{mirmat1}\begin{bmatrix} z+a & w-2d & c \\ b & z-2a & w+d \end{bmatrix}.\end{equation}
\end{proof}

It is easy to see that  the variety defined by the $2\times 2$-minors of (\ref{mirmat1}) is  isomorphic to the one defined by the maximal minors of the matrix
\begin{equation}\label{mirmat2}\begin{bmatrix} z & w-3d & c \\ b & z-3a & w \end{bmatrix} \end{equation}
under the transformation $(z,w)\mapsto (z+a,w+d)$. In what follows we will refer to either of them as \textit{Miranda's matrix form}.  Furthermore, we will take $X=(\mathbb{C}^2,0)$ and show that $Y$ corresponds to an RTP-singularity for the appropriate choices of $a,b,c,d\in \mathcal{O}_{\mathbb{C}^2,0}$. 

%%%%%%%%%%%%%%%%%%%%%%
\section{Graphs of RTP-singularities}\label{sec-rtp}

An  \textit{RTP-singularity} is a surface singularity which is rational with multiplicity $3$. The RTP-singularities are of 9 types and defined by 3 equations in ${\mathbb C}^4$. The explicit equations were first calculated by Tjurina in \cite{tjurina} using the minimal resolution graphs given by Artin in \cite{Ar}.
\begin{table}[h!]
\caption{The minimal resolution graphs of RTP-singularities}
{\label{tablo-res}
 	\begin{tabular} {ccc}
	\hline
%A%
\resizebox{201pt}{74pt}{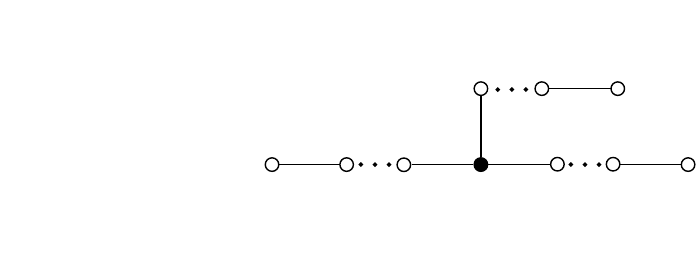} & &
%$B_{k-1,m}$ & 
\resizebox{200pt}{51pt}{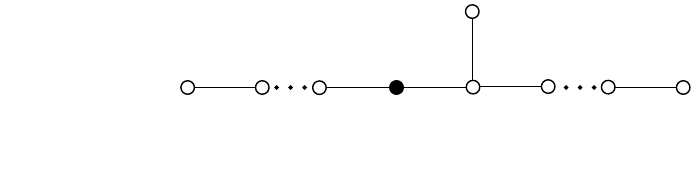}	\\
\noalign{\smallskip}
%$C_{k-1,\ell+1}$ &
\resizebox{203pt}{50pt}{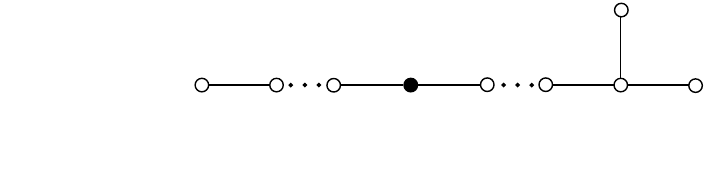}	 & &
%$D_{k-1}$ & 
\resizebox{201pt}{50pt}{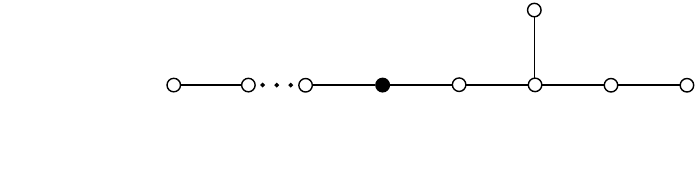}	\\
\noalign{\smallskip}
%$E_{6,0}$ & 
\resizebox{202pt}{33pt}{\begin{picture}(0,0)%
\includegraphics{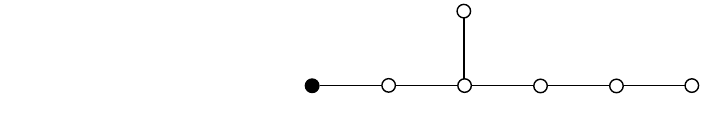}%
\end{picture}%
\setlength{\unitlength}{3552sp}%
\begingroup\makeatletter\ifx\SetFigFont\undefined%
\gdef\SetFigFont#1#2#3#4#5{%
  \reset@font\fontsize{#1}{#2pt}%
  \fontfamily{#3}\fontseries{#4}\fontshape{#5}%
  \selectfont}%
\fi\endgroup%
\begin{picture}(3734,594)(-1814,686)
\put(-1799,764){\makebox(0,0)[lb]{\smash{{\SetFigFont{11}{13.2}{\rmdefault}{\mddefault}{\updefault}{\color[rgb]{0,0,0}$E_{6,0}\colon$}%
}}}}
\end{picture}%
}	& &
%$E_{0,7}$ & 
\resizebox{199pt}{33pt}{\begin{picture}(0,0)%
\includegraphics{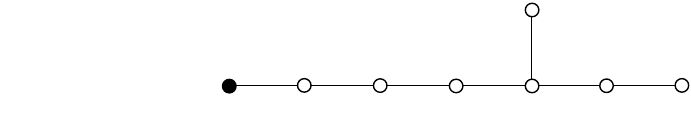}%
\end{picture}%
\setlength{\unitlength}{3552sp}%
\begingroup\makeatletter\ifx\SetFigFont\undefined%
\gdef\SetFigFont#1#2#3#4#5{%
  \reset@font\fontsize{#1}{#2pt}%
  \fontfamily{#3}\fontseries{#4}\fontshape{#5}%
  \selectfont}%
\fi\endgroup%
\begin{picture}(3681,600)(-1364,686)
\put(-1349,764){\makebox(0,0)[lb]{\smash{{\SetFigFont{11}{13.2}{\rmdefault}{\mddefault}{\updefault}{\color[rgb]{0,0,0}$E_{0,7}\colon$}%
}}}}
\end{picture}%
}	\\
\noalign{\smallskip}
%$E_{7,0}$ & 
\resizebox{200pt}{33pt}{\begin{picture}(0,0)%
\includegraphics{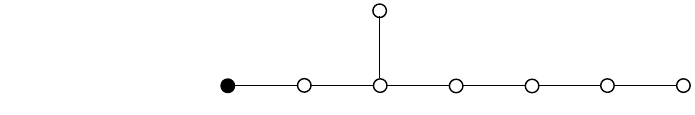}%
\end{picture}%
\setlength{\unitlength}{3552sp}%
\begingroup\makeatletter\ifx\SetFigFont\undefined%
\gdef\SetFigFont#1#2#3#4#5{%
  \reset@font\fontsize{#1}{#2pt}%
  \fontfamily{#3}\fontseries{#4}\fontshape{#5}%
  \selectfont}%
\fi\endgroup%
\begin{picture}(3689,596)(-1364,686)
\put(-1349,764){\makebox(0,0)[lb]{\smash{{\SetFigFont{11}{13.2}{\rmdefault}{\mddefault}{\updefault}{\color[rgb]{0,0,0}$E_{7,0}\colon$}%
}}}}
\end{picture}%
}	& &
%$F_{k-1}$ & 
\resizebox{205pt}{50pt}{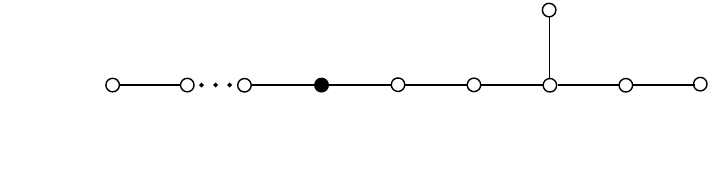}	\\
\noalign{\smallskip}
%$H_n$ & 
\resizebox{201pt}{49pt}{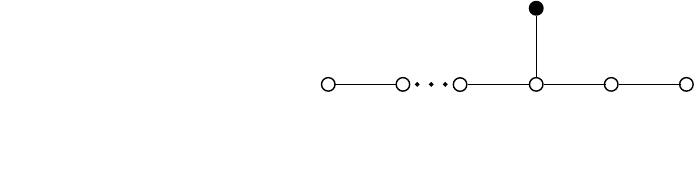}	& & \\\hline \end{tabular}} \end{table}

%\vskip24pt 
The classification of Artin is listed in Table \ref{tablo-res} where the labels $E_{6,0}$, $E_{0,7}$ and $E_{7,0}$ are taken from \cite{tan-etal} and the rest of them from Tjurina's work.  
The equations given by Tjurina are beautiful examples to Miranda's construction of triple covers. Those equations can be obtained by taking $2\times 2$-minors of the matrices listed in the second column of Table \ref{Tlist} below. The calculations required to transform the matrices into Miranda's form, which are shown in the third column of Table \ref{Tlist}, are given in Appendix B. 

\begin{rem} There is a direct way to calculate the minimal resolution graphs of RTP-singularities from Tjurina's equations due to Tjurina herself. Namely, let $(S,0)$ be an RTP-singularity given by the maximal minors of the matrix 
$$\begin{bmatrix} f_1 & f_3 & f_5 \\ f_2 & f_4 & f_6 \end{bmatrix} $$ and consider the embedding $S'\subset \mathbb{C}^4\times\mathbb{P}^1$ 
defined by the equations
$$tf_1=sf_2, tf_3=sf_4,\ tf_5=sf_6$$
where $(s : t)$ are homogeneous coordinates in $\mathbb{P}^1$ (\cite[\S 2]{tjurina}). The surface $S'$ is called the \textit{Tjurina modification} of $(S,0)$ after \cite{tjurina}. It is locally a complete intersection singularity and all of its singularities are rational. The map 
$\varphi\colon (S',E_0) \rightarrow (S,0)$, induced by the projection $\mathbb{C}^4\times\mathbb{P}^1\rightarrow \mathbb{C}^4$, is birational and its fibre above the singular point $0$ is the \textit{central curve} $E_0\cong \mathbb{P}^1$ which corresponds to the exceptional curve with self intersection $-3$ in the minimal resolution graph. Moreover, the RDP-singularities connected to $E_0$ in the minimal resolution graph are the same type of singularities the surface $(S',E_0)$ has at its singular points along $E_0$. Hence, one can deduce the minimal resolution graph of $(S,0)$ by successive blow-ups of $(S',E_0)$. 
\end{rem}

%another easy way giving the resolution of an RTP-singularity $(S,0)$ called \textit{Tjurina modification} (\cite[\S 2]{tjurina}) $\bar{S}\subset \mathbb{C}^4\times\mathbb{P}^1$ defined by the equations
%$$tf_1=sf_2, tf_3=sf_4,\ tf_5=sf_6$$
%where $(s : t)$ are homogeneous coordinates in $\mathbb{P}^1$ and  
%$$\begin{bmatrix} f_1 & f_3 & f_5 \\ f_2 & f_4 & f_6 \end{bmatrix}.$$

%%%%%
%\begin{table}[ht!]
%\caption{RTP-singularities}
%{\label{Tlist}
\begin{longtable}{>{\centering\arraybackslash}m{2cm}>{\centering\arraybackslash}m{4.5cm}>{\centering\arraybackslash}m{8cm}}
\caption{Equations of the RTP-singularities}
\label{Tlist} \\
\hline %\toprule
RTP  & Matrices of Tjurina's equations & Matrices in Miranda's form
\\\hline
	\endfirsthead
		\multicolumn{3}{c}%
{{\tablename\ \thetable{}. -- continued from previous page}} \\
 \hline 
 RTP  & Matrices of Tjurina's equations & Matrices in Miranda's form
		\\\hline
 \endhead
 \hline\multicolumn{3}{r}{{Continued on next page}} \\ %\hline
\endfoot
\hline % \hline
\endlastfoot
\noalign{\smallskip}
%\rule{0pt}{5ex}
$A_{k-1,\ell-1,m-1}$ & $\left[\begin{array}{ccc} z & w & y^m \\ y^k & w+y^\ell & x \end{array}\right]$  &
$\left[\begin{array}{ccc} z & w-(x+y^{\ell}+y^m) & y^m \\ -y^k & z-x+y^k & w \end{array}\right]$
\\[0.6cm] \hline
\noalign{\smallskip}
$B_{k-1,2\ell}$ &  $\begin{bmatrix} z & w+y^\ell & xy \\ y^k & x & w \end{bmatrix}$  &
$\begin{bmatrix} z & w+y^\ell  & xy  \\  y^k & z+x  & w-y^{k+1} \end{bmatrix}$\\[0.5cm]
%\rule{0pt}{5ex}
$B_{k-1,2\ell-1}$ &  $\begin{bmatrix} z & w & xy+y^{\ell} \\ y^k & x & w \end{bmatrix}$  &
$\begin{bmatrix} z & w & xy \\ y^k & z+x-y^{\ell-1} & w-y^{k+1} \end{bmatrix}$\\[0.5cm] \hline
\noalign{\smallskip}
$C_{k-1,\ell+1}$& $\begin{bmatrix} z & w & x^\ell +y^2 \\ y^k & x & w \end{bmatrix}$ &
$\begin{bmatrix} z & w+\ell x^{\ell -1}y^k & x^\ell +y^2 \\ (1-{\ell \choose 3}x^{\ell-3}y^{2k})y^k & z+x & w \end{bmatrix}$ \\[0.5cm]  \hline
\noalign{\smallskip}
$D_{k-1}$ &  $\begin{bmatrix} z & w+y^2& x^2 \\ y^k & x & w \end{bmatrix}$  &
$\begin{bmatrix} z & w+y^2+2xy^k & x^2 \\ y^k & z+x-y^{2k} & w \end{bmatrix}$
\\[0.6cm] \hline
\noalign{\smallskip}
$E_{6,0}$ &  $\begin{bmatrix} z & w & x^2 \\ y & z & w+y^2\end{bmatrix}$ &
$\begin{bmatrix} z & w-y^2 & x^2 \\ y & z & w\end{bmatrix}$ \\[0.6cm] \hline
\noalign{\smallskip}
$E_{0,7}$  &  $\begin{bmatrix} z & w & x^2+y^3 \\ y & z & w\end{bmatrix}$ &
$\begin{bmatrix} z & w & x^2+y^3 \\ y & z & w\end{bmatrix}$\\[0.6cm] \hline
\noalign{\smallskip}
$E_{7,0}$ &  $\begin{bmatrix} z & w & y^2 \\ y & z & w+x^2\end{bmatrix}$  &
$\begin{bmatrix} z & w-x^2 & y^2 \\ y & z & w\end{bmatrix}$\\[0.6cm]  \hline
\noalign{\smallskip}
$F_{k-1}$ &  $\begin{bmatrix} z & w& x^2+y^3 \\ y^k & x & w \end{bmatrix}$ &
$\begin{bmatrix} z & w+2xy^k & x^2+y^3 \\ y^k & z+x-y^{2k} & w \end{bmatrix}$ \\[0.6cm] \hline
\noalign{\smallskip}
$H_{3k-1}$ & $\begin{bmatrix} z & w & xy+y^k \\ x & z & w \end{bmatrix}$ &
 $\begin{bmatrix} z & w & xy+y^k \\ x & z & w \end{bmatrix}$\\
\noalign{\smallskip}
$H_{3k}$ & $\begin{bmatrix} z & w & xy \\ x & z & w+y^k\end{bmatrix}$ &
$\begin{bmatrix} z & w-y^k & xy \\ x & z & w\end{bmatrix}$ \\
\noalign{\smallskip}
$H_{3k+1}$  &  $\begin{bmatrix} z & w+y^k & x \\ xy & z & w\end{bmatrix}$ &
$\begin{bmatrix} z & w+y^k & x \\ xy & z & w\end{bmatrix}$
\\[0.6cm] \hline
\end{longtable}%}\end{tab le}

%\newpage 
%%%%%%
\subsection{Nonisolated forms of RTP-singularities}
%%%%%%%%%%%%%

Now we aim to find hypersurface singularities such that their normalisations are the RTP-singularities. In this setting, normalisation maps will be projections. Recall that a  \textit{generic} projection of a surface $(S,0)\subset (\mathbb{C}^N,0)$ is the restriction of a finite map $p'\colon  \mathbb{C}^N \rightarrow \CC^{3}$ such that its kernel is transversal to the tangent cone of $S$ at $0$ and its degree equals the multiplicity of $S$ at $0$.  By Theorem 4.2.1 of \cite{Le-Teissier}, there exists a Zariski dense open subset $U$ of the space of generic linear projections $(\mathbb{C}^N,0)\rightarrow (\mathbb{C}^3,0)$ such that for every $p'\in U$,
the image $(X,0)$ of $(S,0)$ is a reduced hypersurface and the induced map $p\colon (S,0)\rightarrow (X,0)$ is finite and bimeromorphic.

Now assume that $(S,0)\subset (\mathbb{C}^4,0)$ is a surface defined by the ideal $I=(F,G,H)$ where $F$, $G$ and $H$ are as in (\ref{defY}) with $a,b,c,d\in \mathcal{O}_{\mathbb{C}^2,0}$. Then a generic projection $p$ of $S$ into $\mathbb{C}^3$ can be chosen to be the restriction of the cartesian projection $(x,y,z,w)\mapsto (x,y,z)$ to $S$. Obviously,  the tangent cone of $S$ at $0$ is given by $(z^2+\cdots,zw+\cdots,w^2+\cdots)$ and the kernel is $\{x=y=z=0\}$.

By eliminating the variable $w$, we find that the image of $(S,0)$ in $\mathbb{C}^3$ is the hypersurface
$$(X,0):=\{z^3+3(bd-a^2)z+(3abd-2a^3-b^2c)=0\}.$$
Moreover, the degree of $p$ is $2$ since $\textnormal{dim}_\mathbb{C}\mathcal{O}_{(S,0)}/ p^*\mathfrak{m}_{\mathbb{C}^3,0}=2$. 

On the other hand, the image of a linear projection of $S$ with an RTP-singularity  is a surface $(X',0)$ in $\CC^3$ defined by an equation of the form
\begin{equation}\label{image-surf} z^\nu+f_1(x,y)z^{\nu-1}+\cdots +f_\nu(x,y)=0 \end{equation}
for some positive integer $\nu<\infty$. The following proposition shows that we actually have $\nu=3$. 
Here we refer to a generic projection giving the nonisolated form of an RTP-singularity as \textit{suitable} if the equation gives the expected minimal resolution graph by Oka's process. 

\begin{prop}\label{prop-hyper} A (suitable) projection of each of the Tjurina's equations into $\mathbb{C}^3$ is one of the nonisolated hypersurface given by the equations $(i)-(ix)$ below. 

\begin{enumerate}
\item $A_{k-1,\ell-1,m-1}$, $k,\ell,m\geq 2$.
    \begin{itemize}
    \item $k\geq \ell \geq m$,
    $$z^3+xz^2-(x+y^k+y^\ell+y^m)y^kz+y^{2k+\ell}=0,$$
    \item $k=\ell < m$,
    $$z^3+(x-y^k)z^2-(x+y^{k}+y^m)y^kz+y^{2k+m}=0.$$
        \end{itemize}
\item $B_{k-1,m}$, $k,m\geq 2$.
	\begin{itemize}
	\item $m=2\ell$, %$\ell \leq k+1$ !! diger indisler icin baska denklem gerekli ; 
	$$z^3+xz^2-(y^{k+1}+y^{\ell})y^kz-xy^{2k+1}=0,$$
	\item $m=2\ell-1$, %$\ell \leq k+1$ !! diger indisler icin baska denklem gerekli;  
	$$z^3+(x-y^{\ell-1})z^2-y^{2k+1}z-xy^{2k+1}=0.$$
	\end{itemize}
	%%%%%%%%%%%
\item $C_{k-1,\ell +1}$, $k,\ell\geq 2$,
$$z^3+xz^2-\ell x^{\ell-1}y^{2k}z-(x^\ell+y^2)y^{2k}=0.$$
%%%%%%%%%%%%%% 
\item $D_{k-1}$, $k\geq 2$,
$$z^3+(x+y^{2k})z^2+(2xy^{k}-y^2)y^kz+x^2y^{2k}=0.$$
\item $E_{6,0}$, $$z^3+y^3z+x^2y^2=0,$$
\item $E_{0,7}$, $$z^3+y^5+x^2y^2=0,$$
\item $E_{7,0}$, $$z^3+x^2yz+y^4=0.$$
\item $F_{k-1}$, $k\geq 2$,
$$z^3+(x+y^{2k})z^2+2xy^{2k}z+(x^2+y^{3})y^{2k}=0.$$
\item $H_n$, $n\geq 1$,
	\begin{itemize}
	\item $n=3k-1$; $$z^3+x^2y(x+y^{k-1})=0,$$
	\item $n=3k$; $$z^3+xy^kz+x^3y=0,$$
	\item $n=3k+1$; $$z^3+xy^{k+1}z+x^3y^2=0.$$
	\end{itemize}
\end{enumerate}
\end{prop}

\begin{proof} Consider the equations obtained from the matrices in Miranda's form in Table \ref{Tlist}
and the natural projection $(x,y,z,w)\mapsto (x,y,z)$.  Only the hypersurface equation for $A_{k-1,\ell-1,m-1}$, in the case $k\geq \ell \geq m$, requires an extra transformation of the form $x\mapsto x-y^k$ before the projection.
\end{proof}

\begin{rem} There are many projections one can apply to Tjurina's equations. However, not all of them have cubic surfaces as images. For example, the image of $F_5$ under the projection $(x,y,z,w)\mapsto (x,y,w)$, which is not finite, is an isolated singularity given by $\{-w^2+xy^3+x^3=0\}$. 
\end{rem}

\begin{rem} There also exist cubic hypersurfaces in $\mathbb{C}^3$ which are not rational. For instance, the image of the series $H_k\colon (\mathbb{C}^2,0)\rightarrow (\mathbb{C}^3,0)$, $(x,y)\mapsto (x,y^3,xy+y^{3k-1})$, which is from Mond's classification in \cite{mond87}, is the variety  $\{Z^3-3XY^kZ-X^3Y-Y^{3k-1}=0\}$. Its singular locus is also 1-dimensional but the surface is not rational. Its normalisation is, in fact, $(\mathbb{C}^2,0)$.
\end{rem}

The projection in $\mathbb{C}^3$ of the normal surface singularities is very useful to understand the deformations of normal surface singularities (see, for example, \cite{jongvan}).
In a forthcoming paper, we will show that a rational singularity of multiplicity $m\geq 4$ can be  written as (\ref{image-surf}) with $\nu=m$ by some projection.

%Hence we call nonisolated forms of RTP-singularities the equations given in proposoiton \ref{prop-hyper}. These are of best examples of complete intersection singularities with nonisolated singularities having nice properties. 

%%%%%%%%%%%%%%%%%
\subsection{Resolution of nonisolated forms by Newton polygons}

%The list of the minimal resolution graphs of RTP-singularities in \cite{Ar} were classified by using the properties of some effective divisors supported on the exceptional fiber of the minimal resolution. Later they appeared in \cite{tjurina} with the defining equations given in Table \ref{Tlist} above. As far as we know those were the first understandable examples which are not hypersurface singularities nor complete intersection singularities. Here we refer to the projection giving the nonisolated forms of RTP-singularities above as \textit{suitable} since those equations will give the expected minimal resolution graphs by Oka's process. 
In order to construct the minimal resolutions of RTP-singularities using Oka's theory (\cite{oka}), we start by recalling some notions needed. For details see \cite{ewald} or \cite{fulton}. 

Let $M$ be an integral lattice of rank $n$ with the standard basis $e_1,\ldots,e_n$ and $N$
be its dual integral lattice. Let $M_{\R}:= M \otimes_\mathbb{Z} \R$ and $N_{\R}:= N \otimes_\mathbb{Z} \R$ be the corresponding real vector spaces.
 We will refer to the points of $N$ as \textit{integral} vectors and a vector $\textbf{u}=(u_1,\ldots,u_n)\in N_\mathbb{R}$ as \textit{primitive} if all of its coordinates $u_i$ are coprime.

A nonempty subset $\sigma$ of $N_\mathbb{R}$ is called a \textit{cone} if $\alpha\cdot \textbf{u}\in \sigma$ for all $\textbf{u}\in \sigma$ and $\alpha\in \mathbb{R}$. A \textit{convex polyhedral
cone} is the positive span of a finite set of vectors $\textbf{u}_1,\ldots, \textbf{u}_k\in N_\mathbb{R}$; that is,
 $$\sigma  := \lbrace \sum_{i=1}^k{\lambda_i\textbf{u}_i} \mid  \textbf{u}_i \in  N_{\R}, \lambda_i\in \R_{\geq 0} \rbrace .$$
In this case, we say that $\sigma$ is generated by $\textbf{u}_1,\ldots, \textbf{u}_k$. A convex polyhedral cone generated only by integral vectors is called \textit{rational}
and \textit{strongly convex} if $\sigma \cap (-\sigma)= \{0\}$.
The \textit{dimension} of a cone $\sigma$ is the dimension of the linear space $\R \cdot \sigma$. %The \textit{relative interior} of $\sigma$, denoted by $\textnormal{relint}(\sigma)$, is the topological interior of $\sigma$ as a subset of $\R \cdot \sigma$.

The \textit{dual} of a convex polyhedral cone $\sigma$ is
$$\check{\sigma}=\{\textbf{v} \in M_{\R} \mid \langle \textbf{u},\textbf{v} \rangle \geq 0, \ \forall \textbf{u}\in \sigma \}.$$ A \textit{face} $\tau $ of $\sigma$ is defined by $$\tau := 
\{\textbf{u} \in \sigma \mid  \langle \textbf{u},\textbf{v} \rangle = 0, \textbf{v} \in \check{\sigma}\cap M \} $$

In the rest of this section, $\sigma$ and the word ``cone" will refer to a strongly convex rational polyhedral cone. We will also denote a cone with generators $\textbf{u}_1,\ldots,\textbf{u}_k$ by
$\sigma_{\textbf{u}_1\ldots \textbf{u}_k}$ when we want to emphasise on the generators.

\begin{defn}	
	The \textit{determinant} of a cone $\sigma=\sigma_{\textbf{u}_1\ldots\textbf{u}_k}$ with $\textbf{u}_i=(u_{1i},\ldots,u_{ni})$, $i=1,\ldots ,k$, is the greatest common divisor of $(k\times
k)$-minors of the $(n\times k)$-matrix $U=(u_{ij})$ and denoted by $\textnormal{det}(\sigma)$.
\end{defn}

\begin{defn}  A cone $\sigma$ is \textit{regular}  if its determinant  is equal to $\textnormal{det}(\sigma)=\pm1$.
\end{defn}

\begin{defn}
		A \textit{ fan} $\mathcal P$ of dimension $n$ in $N_{\R}$ is a finite family of $n$-dimensional cones such that
    \begin{enumerate}
		\item Each face of a cone is also a cone in $\mathcal P$,
		\item Any intersection of two cones in $\mathcal P$ is a face of the two cones.
	\end{enumerate}
\end{defn}

\begin{ex}\label{ex-fan} Consider the three vectors $\textbf{u}=(5,4,6)$, $\textbf{v}=(1,0,2)$ and $\textbf{w}=(0,3,2)$ in $\R^3$. The  fan $\mathcal P$ consisting of the 
cones $\sigma_{\textbf{u}\textbf{e}_1\textbf{v}}$, $\sigma_{\textbf{u}\textbf{v}\textbf{e}_3\textbf{w}}$ and $\sigma_{\textbf{u}\textbf{w}\textbf{e}_2\textbf{e}_1}$ 
is pictured in Figure
\ref{fig:N/P/DNP_E_6}a. Its section by the hyperplane $\{x+y+z=1\}$ is also drawn in Figure \ref{fig:N/P/DNP_E_6}b.
\end{ex}

\begin{figure}[h]
\begin{center}$
\begin{array}{cccc}
\resizebox{104pt}{144pt}{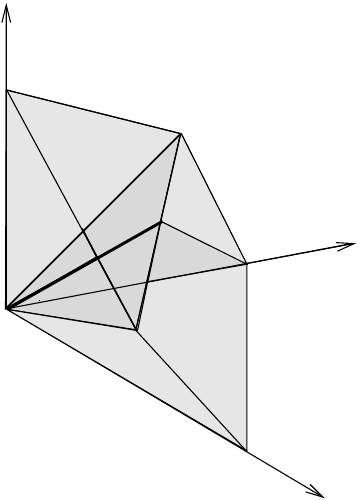}& & &
\resizebox{165pt}{143pt}{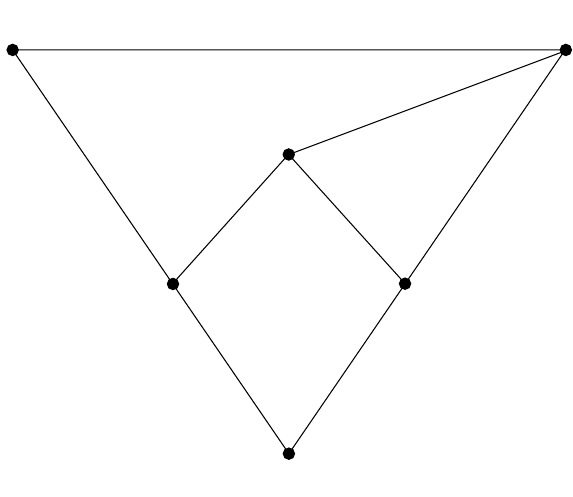}\\ (a) & & & (b)
\end{array}$
\end{center}
\caption{Fan formed by three cones of dimension $3$ in $\R^3$}
\label{fig:N/P/DNP_E_6}
\end{figure}

\begin{defn} A fan $\mathcal P$ is called \textit{regular} if $\textnormal{det}(\sigma)=\pm 1$ for all $\sigma\in \mathcal P$. \end{defn}

\begin{ex}
The fan $\mathcal P$ in Figure \ref{fig:N/P/DNP_E_6}a is not  regular because $\textnormal{det}(\sigma_{\textbf{u}\textbf{e}_1\textbf{v}})=8$.
\end{ex}

\begin{rem}
 Suppose that $\textbf{u}=(1,0,\cdots,0)$ and $\textbf{v}=(v_1,\ldots,v_n)$ two primitive integral vectors generating the cone $\sigma_{\textbf{u}\textbf{v}}$. 
 Then $s:=\textnormal{det}(\sigma_{\textbf{u}\textbf{v}})=\textnormal{gcd}(v_2,\ldots,v_n)$ and $\textnormal{gcd}(s,v_1)=1$ as $\textbf{v}$ is primitive. Let $\textbf{t}\in \sigma_{\textbf{u}\textbf{v}}$ with 
 $\textnormal{det}(\sigma_{\textbf{u}\textbf{t}})=1$. Then there exist positive rational numbers $\alpha,\beta$ such that $\textbf{t}=\alpha\textbf{u}+\beta\textbf{v}$.
 Observe that $\textnormal{det}(\sigma_{\textbf{u}\textbf{t}})=\textnormal{gcd}(\beta v_2,\ldots,\beta v_n)=\beta s$ and $\textnormal{det}(\sigma_{\textbf{t}\textbf{v}})=\alpha s$ are integers.
 Thus the assumption $\textnormal{det}(\sigma_{\textbf{u}\textbf{t}})=1$ implies that $\textbf{t}=\frac{\textbf{v}+\alpha s\cdot\textbf{u}}{s}$.
\end{rem}

More generally, if a given fan is not regular it can be made regular by a process called \textit{regular subdivision}.

%%%%%%%%%%%%%%%%%%%%%%%%%%%%%%
\subsubsection{Regular subdivision of a 2-dimensional cone}
%%%%%%%%%%%%%%%%%%%%%%%%%%%%%%

Let $\sigma:=\sigma_{\textbf{u}\textbf{v}}$ be a $2$-dimensional cone generated by two primitive vectors
$\textbf{u}, \textbf{v}$ in $\N_{\R}$.  Assume that $s_0:=\textnormal{det}(\sigma) >1$. Then there exists a unique integral vector $(\textbf{uv})_1\in \sigma$ \textit{between} $\textbf{u}$ and
$\textbf{v}$ defined by
$(\textbf{uv})_1=\frac{\textbf{v}+s_1\cdot \textbf{u}}{s_0}$  with $s_1:= \textnormal{det}((\textbf{uv})_1,\textbf{v}) \in \ZZ_{\geq 0}$, $1\leq s_1 < s_0$ and satisfying
$\textnormal{det}(\textbf{u},(\textbf{uv})_1) =1$.

\begin{propdef}[\cite{oka}]\label{def-reg} A \textit{regular subdivision} of $\sigma$ is the finite decomposition
$\{(\mathbf{uv})_0=\mathbf{u},(\mathbf{uv})_1,\ldots,(\mathbf{uv})_\alpha,(\mathbf{uv})_{\alpha+1}=\mathbf{v}\}$ where $(\mathbf{uv})_{i+1}\in \sigma_{(\mathbf{uv})_i\mathbf{v}}$
and is determined by the formula
$$(\mathbf{uv})_{i+1}=\frac{\mathbf{v}+s_{i+1}\cdot (\mathbf{uv})_i}{s_i}$$
and  the conditions $s_{i+1} := \textnormal{det}((\mathbf{uv})_i,\mathbf{v})\in \ZZ_{\geq 0}$,  $1<s_{i+1}< s_{i}$ for all $i=0,\ldots, \alpha-2$, and $s_\alpha=1$. Then we have
$\textnormal{det}((\mathbf{uv})_i,(\mathbf{uv})_{i+1})=1$ for all $i=0,\ldots, \alpha$.
\end{propdef}

\begin{ex}
\label{ex:cone1}
 Let $\sigma$ be a $2$-dimensional cone generated by $\textbf{u}=(5,4,6)$ and $\textbf{v}=(1,0,2)$ in $\R^3$ (see Example \ref{ex-fan}). We have $s_0=\textnormal{det}(\sigma)=\textnormal{gcd}(8,4,4)=4 > 1$. We
find
 \begin{alignat*}{1}
 (\textbf{uv})_1 & =  \frac{\textbf{v}+s_1\cdot \textbf{u}}{s_0}=(4,3,5), \hskip26pt s_1=3, \\
 (\textbf{uv})_2 & = \frac{\textbf{v}+s_2\cdot (\textbf{uv})_1}{s_1}=(3,2,4), \ s_2=2, \\
 (\textbf{uv})_3 & =  \frac{\textbf{v}+s_3\cdot (\textbf{uv})_2}{s_2}=(2,1,3), \ s_3=1.
 \end{alignat*}
 Hence the decomposition $\{\textbf{u},(\textbf{uv})_1,(\textbf{uv})_2,(\textbf{uv})_3,\textbf{v}\}$ is the regular subdivision of $\sigma$.
\end{ex}

%%%%%%%%%
\subsubsection{Algorithm for constructing a graph from a 3-dimensional fan}\label{sect-algo} Following Oka's theory in \cite{oka} (see also Section \ref{sect-np}), one can associate a graph to any regular subdivision of a fan as follows.

Let $\mathcal P$ be a  fan in $N_{\R}\cap \R_{\geq 0}^3$ with generators $\textbf{u}_1,\ldots, \textbf{u}_k\in \mathbb{N}^{\oplus 3}$ such that each $\textbf{u}_i$ is a face of a 2-dimensional cone in $\mathcal P$.  Then
\vskip.2cm

\noindent \textbf{Step 1a.} Pick a generator $\textbf{u}:=\textbf{u}_i \in \mathcal P\cap (\mathbb{N}-\{0\})^{\oplus 3}$. 
Consider all 2-dimensional cones $\sigma_{\textbf{u}\textbf{v}_1},\ldots,\sigma_{\textbf{u}\textbf{v}_l}\in \mathcal P$ which are adjacent to $\sigma_{\textbf{u}}$ in $\mathcal P$, i.e. $\sigma_{\textbf{u}\textbf{v}_i}\cap \sigma_{\textbf{u}\textbf{v}_j}=\sigma_\textbf{u}$, $i\neq j$, for $\sigma_{\textbf{v}_j}\in \mathcal P$,
$j=1,\ldots l$.

\noindent \textbf{Step 1b}. For each 2-dimensional cone $\sigma_{\textbf{u}\textbf{v}_j}\in \mathcal P$, $(j=1,\ldots ,l)$, find the regular subdivision 
$\{(\textbf{u}\textbf{v}_j)_0=\textbf{u},\ldots,(\textbf{u}\textbf{v}_j)_{\alpha_i},(\textbf{u}\textbf{v}_j)_{\alpha_i+1}=\textbf{v}_j\}$. 
			
\noindent \textbf{Step 1c}. Construct a tree $\Gamma^j_{\textbf{u}}$ as follows. Assign the vertices  $V_0^{(j)},V_1^{(j)},\ldots,V_{\alpha_i+1}^{(j)}$ to the vectors
$(\textbf{u}\textbf{v}_j)_0,\ldots,$ $(\textbf{u}\textbf{v}_j)_{\alpha_i+1}$ respectively and draw an edge between $V_i^{(j)}$ and $V_{i+1}^{(j)}$ for all 
$i=0,\ldots, \alpha_i$. Note that $V_0^{(j)}=V_0^{(1)}$ since $(\textbf{u}\textbf{v}_j)_0=\textbf{u}$ for all $j$. Let $V_0:=V_0^{(1)}$.
	
\noindent \textbf{Step 1d}.  Erase the vertex $V_{\alpha_j+1}^{(j)}$ and its adjacent edges such that
the graph $\Gamma^j_{\textbf{u}}$ remains connected if its associated vector $\textbf{v}_j$ is not strictly positive. 
 
%\noindent \textbf{Step 1e}. If  $b^j(\textbf{u})=1$ then  $\Gamma^j_{\textbf{u}}:=\Gamma'^j_{\textbf{u}}$. 
%If $b^j(\textbf{u})>1$, take $b^j(\textbf{u})$ copies of $\Gamma'^j_{\textbf{u}}$ and glue them by their corresponding beginning and end points 
%to get a graph. Denote it by $\Gamma^j_{\textbf{u}}$.
	
\noindent \textbf{Step 2}. Glue all $\Gamma^j_{\textbf{u}}$ along the common vertex $V_0$ to obtain an abstract graph $\Gamma_{\textbf{u}}$ corresponding to $\textbf{u}$.

\noindent \textbf{Step 3}. Find the graph $\Gamma_{\textbf{u}_i}$ for each strictly positive generator $\textbf{u}_i$ of $\mathcal{P}$ and glue $\Gamma_{\textbf{u}_i}$ to $\Gamma_{\textbf{u}_j}$ for each $i\neq j $ along their common vertices,
if exists. The resulted connected graph is the graph of the fan $\mathcal P$, denoted by $\Gamma_\mathcal P$ (cf. \cite{oka}).

%%%
\begin{ex}
 \label{ex:cone2} Let $\mathcal{P}$ be the fan studied in Example \ref{ex-fan} and take $\textbf{v}_1:=\textbf{v}$, $\textbf{v}_2:=\textbf{w}$, $\textbf{v}_3:=\textbf{e}_1$. The regular subdivision of
$\sigma_{\textbf{u}\textbf{v}_1}$ given by
$$\{\textbf{u},(\textbf{u}\textbf{v}_1)_1,(\textbf{u}\textbf{v}_1)_2,(\textbf{u}\textbf{v}_1)_3,\textbf{v}_1\}=\{\textbf{u}, (4,3,5), (3,2,4), (2,1,3), \textbf{v}_1\}$$
(see Example~\ref{ex:cone1}). The corresponding graph $\Gamma^1_\textbf{u}$ is then a tree consisting of $5$ vertices as shown in 
Figure \ref{fig:tree1}.

 		The regular subdivision of $\sigma_{\textbf{u}\textbf{v}_2}$ is $\{\textbf{u},(\textbf{u}\textbf{v}_2)_1,(\textbf{u}\textbf{v}_2)_2,\textbf{v}_2\}=\{\textbf{u}, (3,3,4),  (1,2,2),
\textbf{v}_2\}$.
Then $\Gamma^2_\textbf{u}$  is a tree with $4$ vertices (see Figure \ref{fig:tree2}). Finally, the regular subdivision of 
$\sigma_{\textbf{u}\textbf{v}_3}$ is 
$\{\textbf{u},(\textbf{u}\textbf{v}_3)_1,\textbf{v}_3\}=\{ \textbf{u}, (3,2,3), \textbf{v}_3 \} $. So, we get the graph $\Gamma^3_\textbf{u}$ with 3 vertices as
shown in Figure \ref{fig:tree3}.

	 \begin{figure}[h!]
		\centering
		\subfloat[$\Gamma^1_\textbf{u}$]
		{\resizebox{153pt}{25pt}{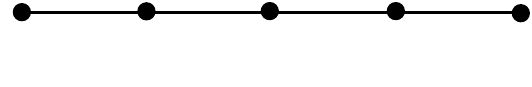}
		\label{fig:tree1}} \quad
		\subfloat[$\Gamma^2_\textbf{u}$]
		{	\resizebox{117pt}{25pt}{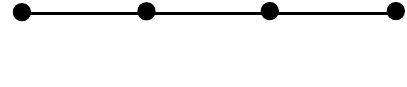}
		\label{fig:tree2}} \quad
		\subfloat[$\Gamma^3_\textbf{u}$]
		{	\resizebox{81pt}{25pt}{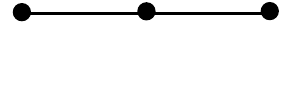}
		\label{fig:tree3}}
		\caption{Examples of trees}
		\label{fig:trees}
 	\end{figure}
Therefore, to construct the graph for $\mathcal{P}$, we delete the vertices $V_4^{(1)}$, $V_3^{(2)}$, $V_2^{(3)}$ which correspond to non-strictly positive vectors $\textbf{v}_1$, $\textbf{v}_2$, $\textbf{v}_3$ respectively. Then we glue  $\Gamma^1_\textbf{u}$, $\Gamma^2_\textbf{u}$ and $\Gamma^3_\textbf{u}$ by overlapping the vertices $V_0^{(1)}=V_0^{(2)}=V_0^{(3)}$. This gives the graph $\Gamma_{\mathcal P}$ shown in Figure \ref{fig:treegraph}.
	 \begin{figure}[h!]
		\centering
		\resizebox{194pt}{65pt}{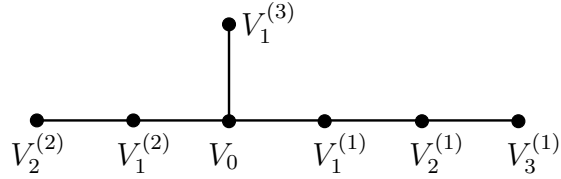}
 		\caption{Graph corresponding to the fan in Example \ref{ex-fan}}
	\label{fig:treegraph}
	\end{figure}

\end{ex}

\begin{rem} One can assign a weight and a genus  to each vertex of any abstract connected graph in a way that the intersection matrix associated with the graph is negative definite. We can obtain a configuration of curves by associating a curve to each vertex in the graph and intersecting any two corresponding curves if there exists an edge between them.  By plumbing construction around such configuration, we can embed the configuration into an analytic surface $\tilde{X}$. As it has the negative definite intersection matrix,  $\tilde{X}$ becomes a resolution of an analytic surface singularity (\cite{grauert}). The graph in Figure \ref{fig:treegraph} with the weight $2$ and genus $0$ assigned to each vertex is the minimal resolution graph of the RDP-singularity of type $E_7$. Note that 
the graph in Figure \ref{fig:treegraph} will also represent the minimal resolution graph of the RTP-singularity of type $E_{6,0}$ (see Table \ref{tablo-res}).  
\end{rem}

We can relate a graph obtained by this process to a hypersurface singularity if we choose the fan $\mathcal{P}$ to be the \textit{dual Newton polygon} of the equation defining the singularity.

%%%%%%%
\subsubsection{Newton Polygon of a singularity}\label{sect-np} 
Let $f(\textbf{z}) := \sum_\textbf{v} a_\textbf{v} \textbf{z}^\textbf{v}\in \mathcal{O}_{\mathbb{C}^n,0}$  be a germ of an analytic function % $f\colon (\mathbb{C}^n,0)  \rightarrow (\mathbb{C},0)$
where $\textbf{z}^\textbf{v} = z_1^{v_1} \cdots  z_n^{v_n}$ with $\textbf{v}=(v_1,\ldots,v_n)\in \mathbb{N}^n$.  The \textit{support} of $f$ is the set
\begin{align*}
	 \supp{f}:=\lbrace \textbf{v}\in\mathbb{N}^n \mid a_\textbf{v} \not= 0  \rbrace.
\end{align*}

\begin{defn}
The Newton polygon $NP(f)$ of $f$ is the boundary of the convex closure in $M_{\R}$ of
	\begin{align*}
		\bigcup_{\textbf{v}\in\supp{f}} \lbrace{\textbf{v}+\mathbb{R}_{\geq 0}^n}\rbrace \subseteq M_{\R}
	\end{align*}
\end{defn}

	\begin{defn}
	  Given a vector $\textbf{u} \in N_\mathbb{R}\cap \R_{\geq0}^n$,  the \textit{face} of $NP(f)$ with respect to $\textbf{u}$ is defined as
		\begin{center}
			$F_{\textbf{u}}:=\lbrace \textbf{v}\in NP(f) \mid \langle \textbf{u},\textbf{v} \rangle=\min_{\textbf{w}\in NP(f)} \langle \textbf{u},\textbf{w} \rangle \rbrace
			\ \subseteq M_{\R}$
		\end{center}
		\label{defn:face}
	\end{defn}
	
Note that $\textbf{u}$ is normal to the face $F_\textbf{u}$. Let us refer to the $(n-1)$-dimensional faces of a polygon as \textit{facets}. Let us define an equivalance relation on $N_\mathbb{R}\cap \mathbb{R}_{\geq0}^n$ by
	$$\textbf{u} \sim \textbf{u}' \textnormal{ if and only if } F_\textbf{u} = F_{\textbf{u}'}.$$
Then each equivalence class forms a cone structure in $N_\mathbb{R}\cap \mathbb{R}_{\geq0}^n$. In fact,
these cones form a fan; it is called \textit{the dual fan} of $NP(f)$ and denoted by $DNP(f)$.  Hence there is a one to one correspondence between the cones of $DNP(f)$ and the faces of
$NP(f)$.

%%%% toric modification

If $f$ is \textit{non-degenerate} (see Definition \ref{defn:nondeg}) then a toric modification associated to a regular subdivision of $NP(f)$ resolves the singularity defined by $f$. In the case of surface singularities, Oka's results provide a canonical way to obtain resolution graph. Before stating his construction we need the following definition.

\begin{defn} Define an integer $g(\textbf{u})$ to be the number of integer points in the interior of the face $F_{\textbf{u}}$ and $r(\sigma_{\textbf{u}\textbf{v}})$ to be the number of integer points in the interior of the face $F_\textbf{u}\cap F_\textbf{v}$ in
$NP(f)$. \end{defn} 

\begin{defn}[Oka's resolution process, \cite{oka}]\label{defn-oka} Let $X$ be a $2$-dimensional hypersurface defined by $f\in \mathcal{O}_{\mathbb{C}^3,0}$ with an isolated singularity at the origin. \textit{Oka's resolution process} for constructing the resolution graph $\Gamma_f$ of a resolution of $X$ consists of two steps. First,  construct $\Gamma_f$ by applying the algorithm in Section \ref{sect-algo} to $DNP(f)$ with an additional operation: glue $r(\sigma_{\textbf{u}\textbf{v}_j})+1$ copies of the trees $\Gamma^j_{\textbf{u}}$ along their beginning vertices. % if $r(\sigma_{\textbf{u}\textbf{v}_j})>0$. 
Secondly, associate weights $w_1,\ldots, w_\alpha$ to each $\Gamma^j_{\textbf{u}}$  defined by the continuous fraction
	\begin{equation}
		[w_1: \ldots :w_\alpha]:=\frac{s}{s_1} = w_1 - \cfrac{1}{\ w_2
	     - \cfrac{1}{\ddots \
	      - \cfrac{1}{w_\alpha}}}
		\label{eqn:continuous/frac}
	\end{equation}
and a weight $w_\textbf{u}$ to $\textbf{u}$ defined by
	\begin{equation}\label{eq-weight}
		w_\textbf{u}:=  \frac{\displaystyle \sum_{j=1}^{l}r(\sigma_{\textbf{u}\textbf{v}_j})\cdot (\textbf{u}\textbf{v}_j)_1}{\textbf{u}}
	\end{equation}
where $l$ is the number of $2$-dimensional cones for which $\textbf{u}$ is one of the generators. 

Then, the vertices corresponding to the vectors $\textbf{u}$ represent the components of $E$ with genus $g(\textbf{u})$ and self intersection numbers $-w_{\mathbf{u}}$, and all the others, i.e. the vertices in $\Gamma_f$ coming from the vectors added in the process of regular subdivision, represent the components of $E$ with genus $0$ and self intersection numbers $-w_i$ calculated by the formula (\ref{eqn:continuous/frac}).
\end{defn}

\begin{rem}If $\mathcal{P}$ is an abstract graph, we can associate weights to all vertices which are introduced by a regular subdivision by the continuous fraction (\ref{eq-weight}). For the other vertices we can deduce the weight by the following fact  for graphs
\begin{eqnarray*}
-a_jw_j+\sum_{i}a_i &\geq& 0, \textnormal{ for all } j,\\
-a_jw_j+\sum_{i}a_i&>&0, \textnormal{ for some } j 
\end{eqnarray*}
where the sum is taken over all $i$ such that the $i$th vertex is connected to the $j$th vertex by an edge. 
The existence of such $a_i$'s are due to Laufer's algorithm (\cite{laufer2}) for the computation of Artin's divisior.
\end{rem}

\begin{ex}\label{ex-dnpf}
	 Let us consider the singularity $E_{6,0}$  given by $f(x,y,z) := z^3+ y^3z+x^2y^2 \in \mathcal{O}_{\mathbb{C}^3,0}$. The support of $f$ is $\supp{f} = \{(2,2,0),(0,3,1),(0,0,3)\}$ and its Newton polygon $NP(f)\subseteq \R^3$ is shown in Figure \ref{fig-NP}.
	\begin{figure}[h!]
		\centering
		\resizebox{212pt}{158pt}{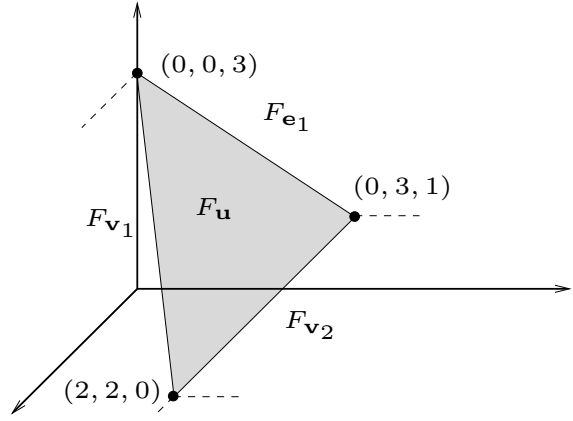}
		\caption{Newton polygon of the nonisolated form of $E_{6,0}$.}
	\label{fig-NP}
	\end{figure}
	
\noindent We see that $NP(f)$ has one compact facet $F_\textbf{u}$ and five non-compact facets $F_{\textbf{e}_1}, F_{\textbf{e}_2}$, $F_{\textbf{e}_3}, F_{\textbf{v}_1}, F_{\textbf{v}_2} $ with the
normal vectors $\textbf{u}=(5,4,6)$, $\textbf{e}_1=(1,0,0)$, $\textbf{e}_2=(0,1,0)$, $\textbf{e}_3=(0,0,1)$ and $\textbf{v}_1=(1,0,2)$, $\textbf{v}_2=(0,3,2)$ respectively. Hence the dual space $DNP(f)$
is the fan given in Fig.~\ref{fig:N/P/DNP_E_6} which was studied in Examples \ref{ex:cone2}. 

The weights for the subdivision of  $\sigma_{\textbf{u}\textbf{v}_1}$ are $[2:2:2]$. For $\sigma_{\textbf{u}\textbf{v}_2}$ the weights are $[3:2]$; and
for $\sigma_{\textbf{u}\textbf{v}_3}$ the weight is $[2]$. The central vertex corresponding to $\textbf{u}$ has  weight $2$. Furthermore, $g(\textbf{u})=0$ and 
$r(\sigma_{\textbf{u}\textbf{v}_j})=0$ for all $j=1,2,3$. 
Therefore, Oka's resolution process applied to the nonisolated form of $E_{6,0}$ yields the graph in Figure \ref{fig:treegraph} which is the minimal resolution graph of $E_{6,0}$ by Artin's classification (see also Table \ref{tablo-res}).
\end{ex}

Note that the value of $r(\sigma_{\textbf{u}\textbf{v}_j})$ can be obtained in a different way as follows.

\begin{defn}\label{def-alt-r} With the notation in Section \ref{sect-algo},  let $\textbf{u}\in \mathcal{P}$ be a strictly positive vector. Let us take the projection of $\mathcal P$ onto the plane whose normal is $\textbf{u}$ to get the vectors $\tilde{\textbf{v}}_1,\ldots,\tilde{\textbf{v}}_l$ which are not necessarily strictly positive. We define the constants $c_1,\ldots,c_l \in \mathbb{N}-\{0\} $ to be the minimal solution of
$\sum c_j\tilde{\textbf{v}}_j=0$. %We call $b^j(\textbf{u})$ as the number of copies of $\Gamma'^j_{\textbf{u}}$.
\end{defn}

By direct calculation, we have

\begin{lem} For an RTP-singularity, we have $c_j=r(\sigma_{\mathbf{u}\mathbf{v}_j})+1$ for all $j$.
\end{lem}

\begin{rem}\label{rem-nonnew} For a fan which is not a dual Newton polygon of a function, the weight formula (\ref{eq-weight}) or Definition \ref{def-alt-r} may not always yield integers. For instance, consider
the fan $\mathcal{P'}\subset \R^3$ pictured in Fig.~\ref{fig:counter}. 
\begin{figure}[h!]
		\centering
		\resizebox{182pt}{157pt}{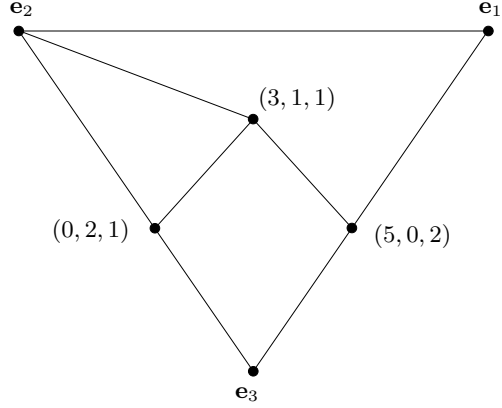}
		\caption{A fan $\mathcal{P}$}
	\label{fig:counter}
	\end{figure}

\noindent Each $2-$dimensional cone in  $\mathcal{P'}$ is regular. So, the related graph $\Gamma_{\mathcal{P}'}$ consists of only one vertex. However, there is no solution to the weight formula $$w_{(3,1,1)}=\frac{\alpha_1(5,0,2)+\alpha_2(0,2,1)+\alpha_3(0,1,0)}{(3,1,1)}$$ with $w_{(3,1,1)},\alpha_1,\alpha_2,\alpha_3\in \mathbb{N}-\{0\}$. Therefore, an abstract fan does not necessarily give a resolution graph. 
\end{rem}

Recall that we have chosen a suitable projection to get the nonisolated forms in order to obtain the minimal resolution graphs for each singularity given by \cite{Ar}. We observed by some elementary but time consuming calculations that the nonisolated forms of RTP-singularities given in  Proposition \ref{prop-hyper} are all non-degenerate singularities (see Appendix A).

\begin{prop} The minimal resolution graphs of nonisolated forms of RTP-singularities defined by the equations $(i)-(ix)$ in Proposition \ref{prop-hyper} can be obtained by Oka's process followed by a number of blow-downs and coincide with Artin's classification in \cite{Ar}.
\end{prop}

\begin{proof} Simply, apply Oka's process to the equations. See Table \ref{tablo1} for the main steps. 
Note that $g(\textbf{u})=0$ for all the series. The integers $r(\sigma)$ are all equal to $0$ except in the following cases. For the series $A_{k-1,\ell-1,m-1}$, $r(\sigma_{\mathbf{u}_1\mathbf{e}_1})=1$  if $k=\ell<m$ or if $k\geq \ell \geq m$ and $k$, $m$ are both even or odd; for $B_{k-1,2\ell}$, $r(\sigma_{\mathbf{u}_1\mathbf{e}_1})=1$; for $C_{k-1,\ell+1}$, $r(\sigma_{\mathbf{u}_1\mathbf{e}_1})=2$ if $k=3p+2$, $r(\sigma_{\mathbf{u}_2\mathbf{e}_3})=1$ if $\ell$ is even and $r(\sigma_{\mathbf{u}_2\mathbf{v}})=1$ if $\ell$ is odd; for $D_{k-1}$, $r(\sigma_{\mathbf{u}_1\mathbf{e}_1})=1$ if $k$ is even; finally, for $F_{k-1}$, $r(\sigma_{\mathbf{u}_1\mathbf{e}_1})=2$ if $k=3p$. 

Moreover, the weight of $\mathbf{u}_1$ is equal to $1$ in the following cases: $A_{k-1,\ell-1,m-1}$ if $k\geq \ell \geq m$ and $k$ is odd, $m$ is even or $k$ is even, $m$ is odd; $B_{k-1,2\ell}$ if  $k$ is odd, $\ell$ is even or $k$ is even, $\ell$ is odd; $D_{k-1}$ if $k$ is odd; $F_{k-1}$ if $k=3p+1$ or $k=3p+2$. Therefore, those require one blow-down after the process. We also find $w_{\mathbf{u}_1}=1$ for the series $C_{k-1,\ell+1}$ when $k=3p$ or $k=3p+1$. However, one needs to apply two successive blow-downs. On the other hand, $w_{\mathbf{u}_3}=1$ for $B_{k-1,2\ell}$ in which case one needs $k-\ell+1$ successive blow-downs to find  the minimal resolution graph. This concludes the proof.
\end{proof}

Let $S$ be an RTP-singularity defined by the maximal minors of one of the matrices in Table \ref{Tlist} and $X$ its nonisolated form given in Proposition \ref{prop-hyper}. Let $\tilde{X}$ be the resolution of $X$ by Oka's process. Then, by contracting the $(-1)$-curves on $\tilde{X}$ using Castelnouva criterion we get the minimal resolution $\tilde{S}$ of $S$. Therefore, we have the commutative diagram
\begin{equation}\xymatrixcolsep{5pc}\xymatrix{% & S \ar@{=}[d] \ar[rd]^p & \\  
 \tilde{S}  \ar[r]^{\pi} & S \ar[r]^p & X  \\
\tilde{X} \ar[rru]_\sigma \ar[u]^e &&
}
\end{equation}
where $\pi$ is the minimal resolution by successive blow-ups, $p$ is the projection which is also  a normalisation, $e$ is the contraction map and $\sigma$ is the resolution of $X$ obtained by Oka's process.

%\begin{equation}\xymatrixcolsep{5pc}\xymatrix{ &  \tilde{X}  \ar[r]^{\pi} \ar[d]_{\cong} & S \ar@{=}[d] \ar[rd]^p & \\  
%\tilde{S}_1 \ar@/_1pc/[rrr]_\sigma \ar[r]^e & \tilde{S}_2 \ar[r] & S \ar[r]^n & X 
%} \end{equation}

As one can expect, many cubic equations may give the same Newton polygon. However, they may not come from a projection of an RTP-singularity.
 For example, the hypersurface 
\begin{equation}\label{deg-cubic} z^3+xz^2-y^{2k+1}z-xy^{2k+1}=0 \end{equation}
 has the same Newton polygon as $B_{k-1,2k+2}$ but its normalisation is smooth; moreover, it is degenerate (see Example \ref{ex-deg}). Therefore it is not an isolated form of an RTP-singularity.

\newpage
 	\begin{longtable} {>{\centering\arraybackslash}m{1.8cm}>{\centering\arraybackslash}m{5cm}>{\centering\arraybackslash}m{5.75cm}>{\centering\arraybackslash}m{1.8cm}}
	\caption{Resolution process of the RTP-singularities.}
\label{tablo1}\\
		\hline 
		RTP 	& $NP(f)$ 	&  a subdivision of $DNP(f)$ & Conditions	
		\\ \hline	
		\endfirsthead
		\multicolumn{4}{c}%
{{\tablename\ \thetable{}. -- continued from previous page}} \\
 \hline 
 RTP 	& $NP(f)$ 	&  a subdivision of $DNP(f)$ & Conditions	 
		\\\hline
 \endhead
 \hline\multicolumn{4}{r}{{Continued on next page}} \\ 
 \endfoot
\hline 
\endlastfoot
 		\noalign{\smallskip}
		$A_{k-1,\ell-1,m-1}$ 		
		&\resizebox{142pt}{101pt}{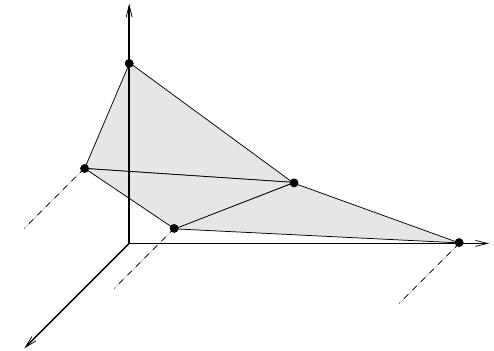}		
		& \resizebox{129pt}{101pt}{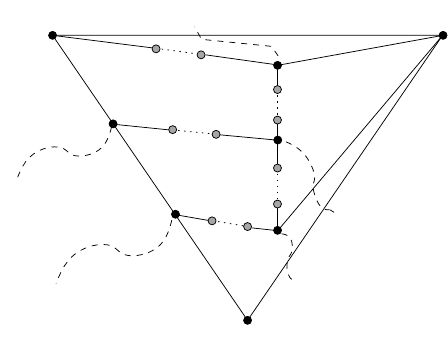}
		& \begin{small}$\natural$\end{small}	
		\cr
		\cline{1-4} \noalign{\smallskip}
		$A_{k-1,\ell-1,m-1}$		
		&\resizebox{144pt}{101pt}{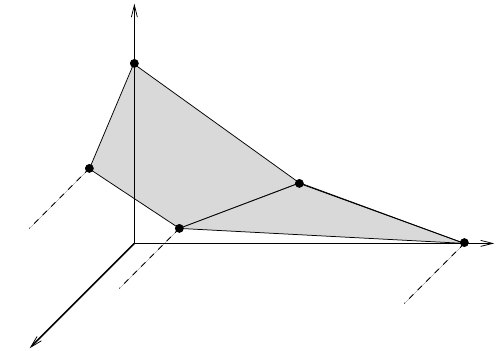}	
		& \resizebox{137pt}{101pt}{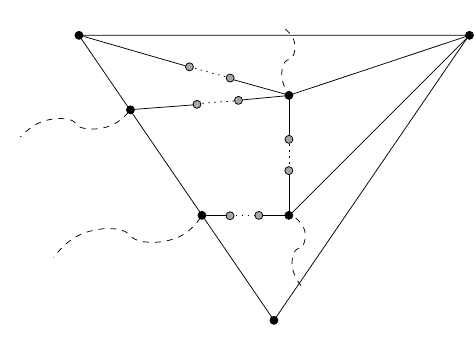}
		& \begin{small}$\sharp$\end{small}
		\\
		\hline\noalign{\smallskip}
		$B_{k-1,2\ell}$		
		&\resizebox{142pt}{101pt}{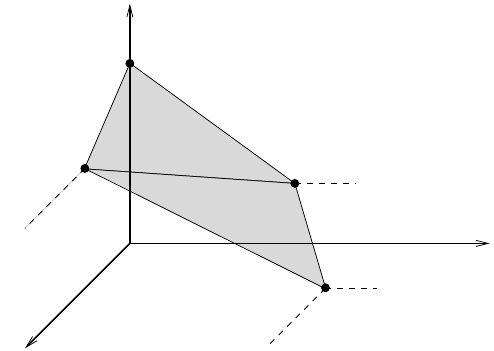}			
		& \resizebox{127pt}{101pt}{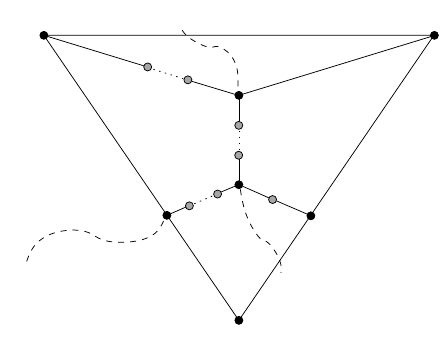}	
		& \begin{small}$\ell\leq k+1$ \& $\flat_1$ \end{small}
		\\
		\cline{1-4}\noalign{\smallskip}
		$B_{k-1,2\ell-1}$	
		&\resizebox{142pt}{101pt}{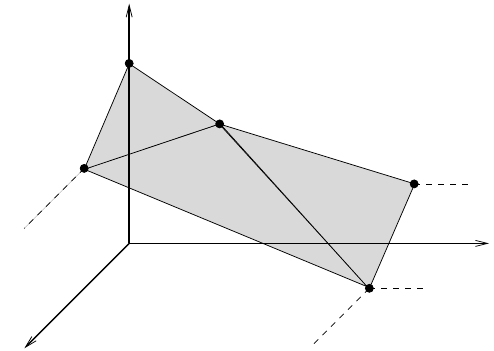}		
		& \resizebox{138pt}{101pt}{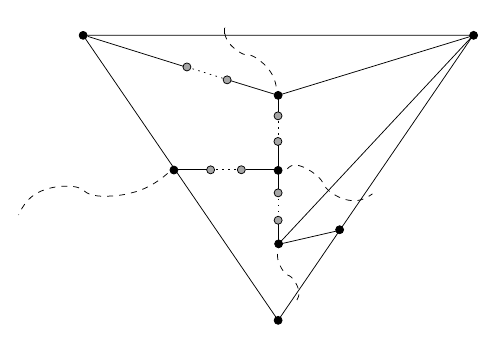}			
		& \begin{small}$\ell\leq k+1$ \& $\flat_2$ \end{small}
		\\
	\hline
\multicolumn{4}{p{15cm}}{\begin{footnotesize}$\natural$: $k\geq \ell \geq m$, $\beta=\ell-m-1$; $ \alpha=\frac{k+m-2}{2}$, , $\gamma=\frac{k-m-2}{2}$ if $k, m$ both even or odd; $\alpha=\frac{k+m-1}{2}$, $\gamma=\frac{k-m-1}{2}$ if $k$ odd, $m$ even or $k$ even $m$ odd.\end{footnotesize}} \cr
\multicolumn{4}{p{15cm}}{\begin{footnotesize}$\sharp$: $k=\ell<m$, $\beta=k-m-1$.\end{footnotesize}}\cr
\multicolumn{4}{p{15cm}}{\begin{footnotesize}$\flat_1$: $\alpha=\frac{k+\ell-2}{2}$, $\gamma=\frac{k-\ell}{2}$ if $k$, $\ell$ both even or odd; $\alpha=\frac{k+\ell-1}{2}$ and $\gamma=\frac{k-\ell+1}{2}$ if $k$ odd, $\ell$ even or $k$ even $m$ odd.\end{footnotesize}}\cr
\multicolumn{4}{p{15cm}}{\begin{footnotesize}$\flat_2$: $\alpha=2k-\ell+2$, $\beta=k-\ell+1$.\end{footnotesize}}
	\\
	\hline\noalign{\smallskip}
		%cift
		$C_{k-1,\ell+1}$		
		&\resizebox{145pt}{101pt}{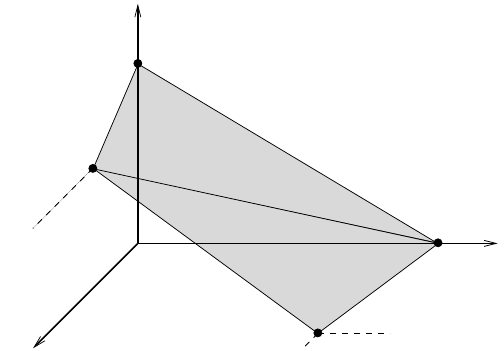}		
		& \resizebox{124pt}{101pt}{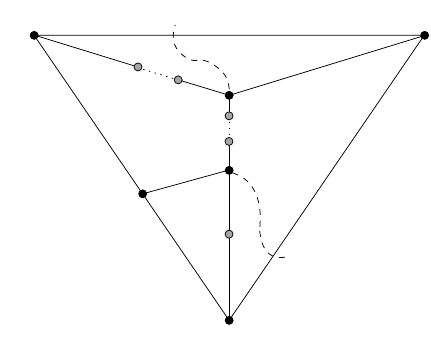}		
		&  \begin{small}$\ell$ even \& $\clubsuit$\end{small}
		\\
		\cline{1-4}\noalign{\smallskip}
		%tek
		$C_{k-1,\ell+1}$			
		&\resizebox{145pt}{101pt}{\input{NpolyCi.pdf_t}}	
		&  \resizebox{124pt}{101pt}{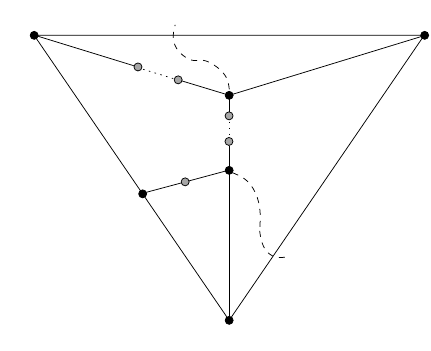}		
		& \begin{small}$\ell$ odd \& $\clubsuit$\end{small}
		\\\hline \noalign{\smallskip}
		$D_{k-1}$		
		&\resizebox{145pt}{101pt}{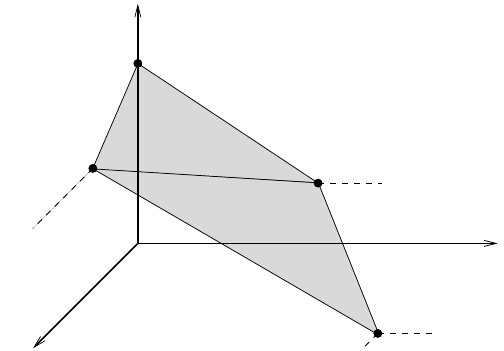}	
		&  \resizebox{128pt}{101pt}{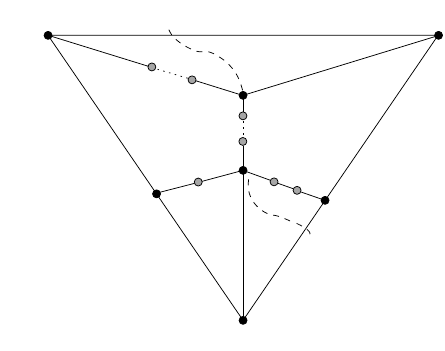}
		& \begin{small}$\spadesuit$\end{small}
		\\ 	\hline\noalign{\smallskip}
		$E_{6,0}$		
		&\resizebox{136pt}{101pt}{\input{NpolyE60.pdf_t}}	
		&  \resizebox{129pt}{101pt}{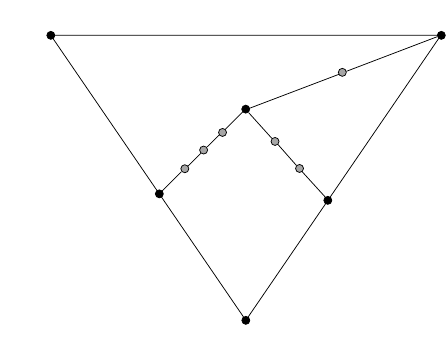}		& 
		\\	\hline \noalign{\smallskip}
		$E_{0,7}$		
		&\resizebox{136pt}{101pt}{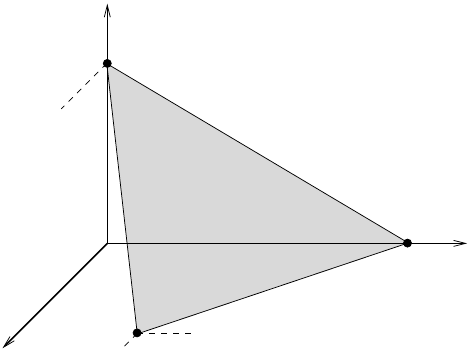}	
		&  \resizebox{129pt}{101pt}{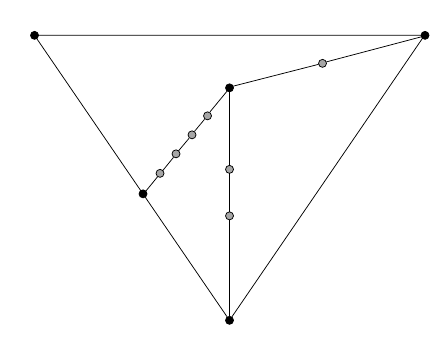}		&
		\\
		\hline
\multicolumn{4}{p{15cm}}{\begin{footnotesize}$\clubsuit$: $\gamma=p+\ell-1$ if $k= 3p+1$; Ê$\gamma=p+\ell-2$ if $k= 3p+2$ or $k=3p$.\end{footnotesize}}\\
\multicolumn{4}{p{15cm}}{\begin{footnotesize}$\spadesuit$: $\gamma=p$ if $k= 2p$; $\gamma=p+1$ if $k= 2p+1$.\end{footnotesize}}
	\\
	\noalign{\smallskip}
	$E_{7,0}$		
	&\resizebox{140pt}{101pt}{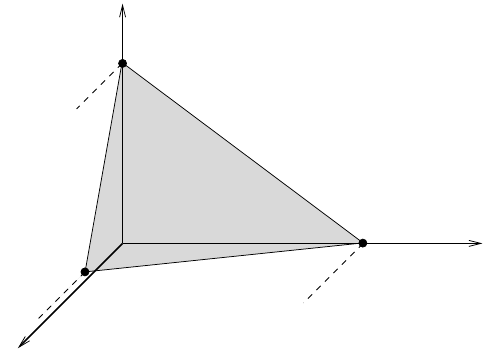}	
	&  \resizebox{117pt}{101pt}{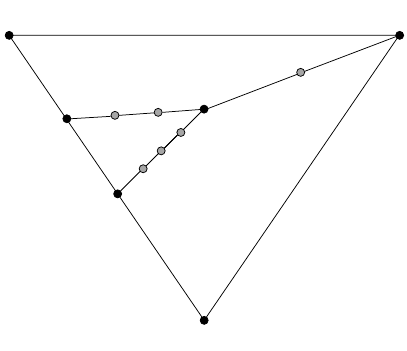}	
	&
	\\
	\hline\noalign{\smallskip}
	$F_{k-1}$		
	&\resizebox{142pt}{101pt}{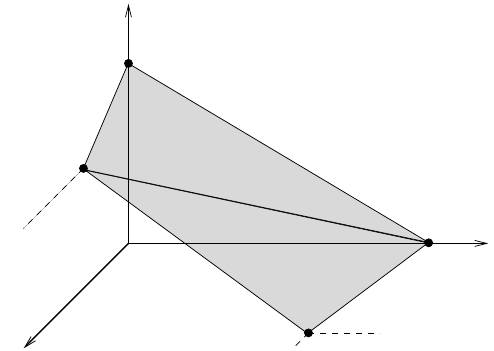}	
	&  \resizebox{124pt}{101pt}{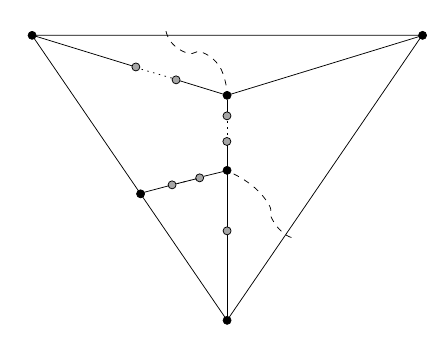}
	&$\diamondsuit$
	\\
\hline\noalign{\smallskip}
	$H_{3k-1}$	
	&\resizebox{135pt}{105pt}{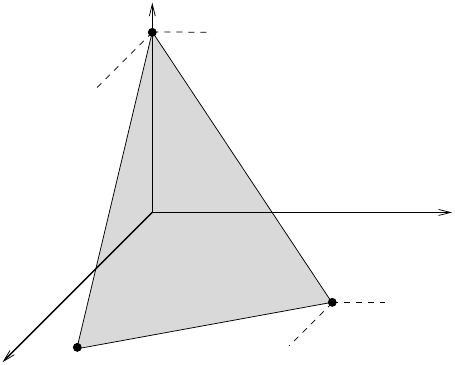}	
	&  \resizebox{127pt}{101pt}{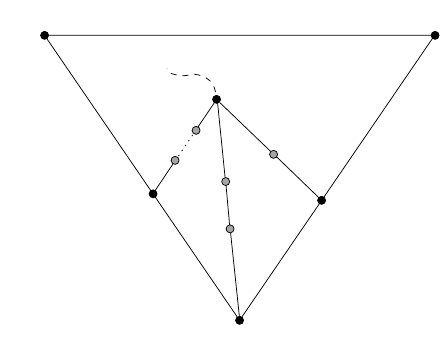}		
	&
	\\
	\cline{1-4}\noalign{\smallskip}
	$H_{3k}$	
	&\resizebox{135pt}{105pt}{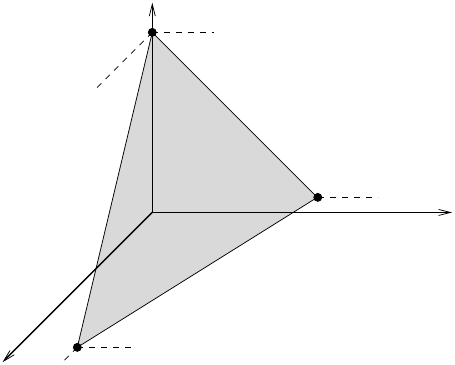}	
	&  \resizebox{129pt}{101pt}{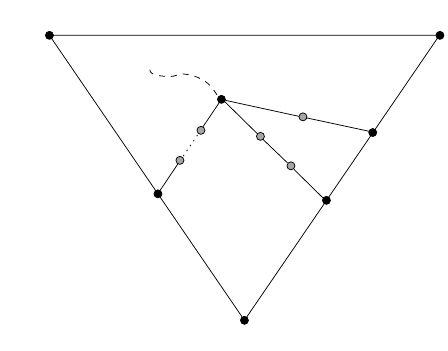}
	&
	\\
	\cline{1-4}\noalign{\smallskip}
	$H_{3k+1}$	
	&\resizebox{135pt}{105pt}{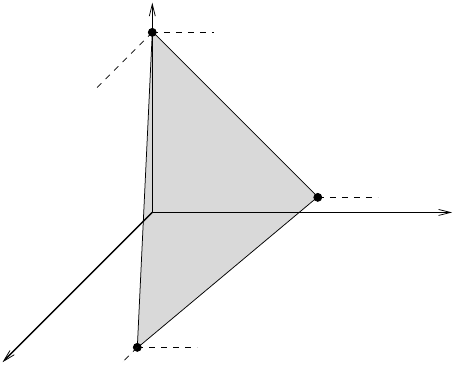}	
	&  \resizebox{127pt}{101pt}{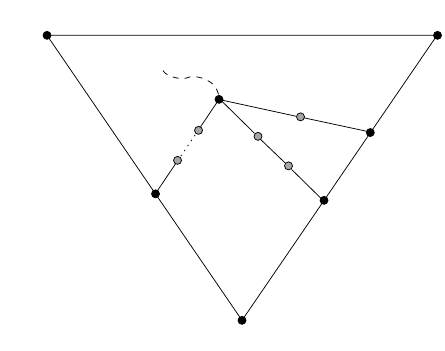}		
	&
	\\
	\hline
	\multicolumn{4}{p{14cm}}{\begin{footnotesize}$\diamondsuit$: $\gamma=p+2$ if $k= 3p+3$ or $k=3p+1$; $\gamma=p+3$ if $k= 3p+2$.\end{footnotesize}}
	\end{longtable}

%%%%%%%%%%%%
\newpage
\appendix
\section{Newton non-degeneracy of RTP-singularities}

\begin{defn}[\cite{oka}]\label{defn:nondeg}
An analytic function $f(\textbf{z}) = \sum_\textbf{v} a_\textbf{v} \textbf{z}^\textbf{v}$ in $\mathcal{O}_{\mathbb{C}^n,0}$ is  non-degenerate with respect to its Newton polyhedron in coordinates 
$(z_1,\ldots,z_n)$ or shortly \textit{Newton non-degenerate}, if for all compact faces $F_\textbf{u}$ associated to a non-zero vector $\textbf{u}$ of $NP(f)$, 
the \textit{face function} or the \textit{initial form}
$\textnormal{In}_{\textbf{u}}(f):= \sum_{\textbf{v}\in F_\textbf{u}}{a_\textbf{v} \textbf{z}^\textbf{v}}$
defines a nonsingular hypersurface in the torus $(\mathbb{C}^*)^n$.

\end{defn}
More generally, for an ideal  $I$ in $\mathcal{O}_{\mathbb{C}^n,0}$, the \textit{initial ideal} is given by $\textnormal{In}_\textbf{u}(I)= \langle \textnormal{In}_\textbf{u}(f) | f\in I \rangle$. And,

\begin{defn}[\cite{aroca}]
An affine variety $V(I)\subseteq \CC^n $ is said to be \textit{Newton non-degenerate} 
if for every $\textbf{u} \in \R_{\geq0}^n$, $V(\textnormal{In}_{\textbf{u}}I)$  does not have any singularity in $(\CC^*)^n$.
\end{defn} 

\begin{ex}\label{ex-nondeg}
 Consider the RTP-singularity $V(F_{k-1})$ defined by  Proposition \ref{prop-hyper} (v). The $NP(f)$ has two compact faces ${F_{\textbf{u}_1}}$ and ${F_{\textbf{u}_2}}$(see Table~\ref{tablo1}).
  Then the face functions of
$f$ are
 $$\textnormal{In}_{\textbf{u}_1}= z^3+xz^2+y^{2k+3}, \ \
 \textnormal{In}_{\textbf{u}_2}= xz^2+x^2y^{2k}+y^{2k+3}.$$
 The Jacobian ideals are
 \begin{eqnarray*} J(\textnormal{In}_{\textbf{u}_1})&=&(2xz,(2k+3)y^{2k+2},3z^2+2xz), \\
 J(\textnormal{In}_{\textbf{u}_2})&=&(z^2+2xy^{2k},(2k+3)y^{2k+2}+(2k)x^2y^{2k-1},2xz).\end{eqnarray*}
Clearly, none of $\textnormal{In}_{\textbf{u}_1}$ and $\textnormal{In}_{\textbf{u}_2}$ has a solution in $(\CC^*)^3$. Hence $F_{k-1}$ is non-degenerate.
\end{ex}

\begin{ex}\label{ex-deg} The hypersurface given by (\ref{deg-cubic}) is degenerate. Its Newton polygon has one compact face. Therefore, the degeneracy of the hypersurface follows from the fact that the Jacobian ideal, which is given by
$$(z^2-y^{2k+1}, y^{2k}z-x^{2k}, 3z^2+2xz-y^{2k+1}),$$ 
has a solution in the torus.
\end{ex}

\begin{rem}
Non-degeneracy of a hypersurface singularity in $\mathbb{C}^3$ can be checked by simple calculations as  in Example \ref{ex-nondeg}. However, 
it might be very useful to work with computer programs such as  GFAN (\cite{gfan}) and Singular (\cite{Sing}) when one studies ideals in higher dimensions. 
In the following example, we give an explicit computation of one of the equations given by Tjurina (see Table~\ref{Tlist}).
\end{rem}

\begin{rem} An explicit computation shows that all of the nonisolated forms of RTP-singularities listed in Proposition \ref{prop-hyper} are Newton non-degenerate. 
\end{rem}

In fact,  

\begin{thm}\label{thm-nondeg} The nonisolated forms of RTP-singularities and their normalisations are Newton non-degenerate.
\end{thm}

Hence we are inspired to suggest the following conjecture.

\begin{conj}
A normal surface singularity is Newton non-degenerate if and only if it is a normalisation of a nonisolated non-degenerate hypersuface singularity. \end{conj}

%there exists a nonisolated non-degenerate hypersurface singularity obtained by a suitable projection.
%the nonisolated hypersurface singularity (obtained by a suitable projection) is non-degenerate.

%\newpage
\section{Tranformations used in Table \ref{Tlist}}
Here we list the diffeomorphisms we apply (in the given order) to transform Tjurina's matrix form of RTP-singularities into  Miranda's given by (\ref{mirmat2}). 

\begin{alignat*}{2}
&A_{k-1,\ell-1,m-1}:  & & \cr
&  & \mathbf{(1)}\  (x,y,z,w)&\mapsto (x,y,z,w-y^\ell) \cr
&  & \mathbf{(2)}\  (x,y,z,w)&\mapsto (x-z-w,y,z,w) \cr
& &  \mathbf{(3)}\  (x,y,z,w)& \mapsto (x,y,z,w+x+y^k+z)   \cr
&B_{k-1,2\ell}: && \cr
& &\mathbf{(1)}\   (x,y,z,w)&\mapsto  (x-z,y,z,w)
\cr
&B_{k-1,2\ell-1}: && \cr
& &\mathbf{(1)}\   (x,y,z,w)&\mapsto  (x-z+y^{\ell-1},y,z,w) \cr
&C_{k-1,\ell+1}: & &\cr
& &\mathbf{(1)}\   (x,y,z,w)&\mapsto (x-z,y,z,w) \cr
&&\mathbf{(2)}\   (x,y,z,w)&\mapsto \left(x,y,z,w- (\ell x^{\ell-1}+{\ell \choose 2}x^{\ell -1}z+\cdots +{\ell \choose \ell-1} xz^{\ell -2}+z^{\ell-1}\right )
\cr
& & \mathbf{(3)}\  (x,y,z,w)&\mapsto \left (x,y,z, (1-{\ell \choose 3}x^{\ell-3}y^{2k})w \right)
\cr
& D_{k-1} \textnormal{ and } F_{k-1}: && \cr
&& \mathbf{(1)}\  (x,y,z,w)&\mapsto (x-z,y,z,w) \cr %\ \textnormal{and }
& & \mathbf{(2)}\  (x,y,z,w)& \mapsto (x,y,z,w-2xy^k-y^kz)\cr
& H_{3k+1}: & &\cr
& &\mathbf{(1)}\  (x,y,z,w)&\mapsto (x,y,z,w+y^k) \cr
& E_{6,0}:  &&\cr
& & \mathbf{(1)}\   (x,y,z,w)&\mapsto (x,y,z,w+y^2) \cr
& E_{7,0}: &&\cr
& & \mathbf{(1)}\   (x,y,z,w)&\mapsto (x,y,z,w+x^2).
\end{alignat*}
Note that additional row and column operations may be needed in some of the cases.
%Thus, all the classes of RTP-singularities can also be obtained by Miranda's construction.

%%%%%%%%%%%%%%%%%%%%%%%%%%%%%%%%%%%%%%%%

\end{document}